\documentclass[12pt]{amsart}
\usepackage{amsmath}	
\usepackage{amssymb}
\usepackage{tikz}
\usetikzlibrary{calc}

\newtheorem{theorem}{Theorem}[section]

\newtheorem*{theoremA}{Theorem A}
\newtheorem*{theoremB}{Theorem B}
\newtheorem*{theoremC}{Theorem C}
\newtheorem{lemma}[theorem]{Lemma}

\newtheorem*{maincorollary}{Corollary}

\theoremstyle{remark}

\theoremstyle{definition}

\newcommand{\et}{\quad\mbox{and}\quad}

\newcommand{\bP}{\mathbb{P}}
\newcommand{\bQ}{\mathbb{Q}}
\newcommand{\bR}{\mathbb{R}}
\newcommand{\bZ}{\mathbb{Z}}

\newcommand{\dist}{\mathrm{dist}}

\newcommand{\tuv}{\uv^\perp}
\newcommand{\tuw}{\uw^\perp}
\newcommand{\ud}{\mathbf{d}}
\newcommand{\uu}{\mathbf{u}}
\newcommand{\uv}{\mathbf{v}}
\newcommand{\uw}{\mathbf{w}}
\newcommand{\ux}{\mathbf{x}}
\newcommand{\uy}{\mathbf{y}}
\newcommand{\uz}{\mathbf{z}}

\newcommand{\disp}{\displaystyle}

\newcommand{\gee}{\ge^*}
\newcommand{\lee}{\le^*}

\begin{document}

\baselineskip=18pt

\title[Diophantine approximation with sign constraints]
{Diophantine approximation with sign constraints}
\author{Damien ROY}
\address{
   D\'epartement de Math\'ematiques\\
   Universit\'e d'Ottawa\\
   585 King Edward\\
   Ottawa, Ontario K1N 6N5, Canada}
\email{droy@uottawa.ca}
\thanks{Work partially supported by NSERC}

\begin{abstract}
Let $\alpha$ and $\beta$ be real numbers such that $1$, $\alpha$ and
$\beta$ are linearly independent over $\bQ$.  A classical result of
Dirichlet asserts that there are infinitely many triples of integers
$(x_0,x_1,x_2)$ such that
$|x_0+\alpha x_1+\beta x_2| < \max\{|x_1|,|x_2|\}^{-2}$.
In 1976, W.~M.~Schmidt asked what can be said under the
restriction that $x_1$ and $x_2$ be positive.  Upon denoting by
$\gamma\cong 1.618$ the golden ratio, he proved that there are triples
$(x_0,x_1,x_2) \in \bZ^3$ with $x_1,x_2>0$ for which the product
$|x_0 + \alpha x_1 + \beta x_2| \max\{|x_1|,|x_2|\}^\gamma$ is
arbitrarily small.   Although Schmidt later conjectured that
$\gamma$ can be replaced by any number smaller than $2$,
N.~Moshchevitin proved very recently that it cannot be replaced by a number
larger than $1.947$.  In this paper, we present a construction showing
that the result of Schmidt is in fact optimal.
\end{abstract}

\maketitle

\section{Introduction}
 \label{intro}

Given $\alpha,\beta\in\bR$, a well-known result of Dirichlet
asserts that, for each $X\ge 1$, there exists a non-zero
point $(x_0,x_1,x_2)\in\bZ^3$ which satisfies
\[
 |x_0+x_1\alpha+x_2\beta|\le X^{-2},
 \quad |x_1|\le X,
 \quad |x_2|\le X
\]
(see for example \cite[Chap.~II, Theorem 1A]{Sc80}).  In 1976, on
the occasion of E.~Hlawska's sixtieth birthday, W.~M.~Schmidt
proved a variation of this result in which the coordinates
$x_1$ and $x_2$ are required to be positive \cite{Sc76}.
Upon denoting by
\[
 \gamma = \frac{1+\sqrt{5}}{2} \cong 1.618
\]
the golden ratio, his result can be stated as follows.

\begin{theoremA}[Schmidt]
Let $\alpha,\beta\in\bR$ and let $\epsilon>0$.  Suppose
that $1,\alpha,\beta$ are linearly independent over $\bQ$.
Then, there are arbitrarily large values of $X$ for which
the inequalities
\[
 |x_0+x_1\alpha+x_2\beta|\le \epsilon X^{-\gamma},
 \quad 0<x_1\le X,
 \quad 0<x_2\le X
\]
admit a solution $(x_0,x_1,x_2)\in\bZ^3$.
\end{theoremA}

In \cite{Sc76}, Schmidt made several comments on this
result.  First, he gave an example showing that it is
false without the condition that $1,\alpha,\beta$ are
linearly independent over $\bQ$.  He also added a remark
about the general situation involving a larger number of
variables $x_0,\dots,x_k$. This was further studied by
P.~Thurnheer in \cite{T90}.  Finally, he asked if the
exponent $\gamma$ can be replaced by a larger number
without invalidating the statement of Theorem A.  Then,
several years later, he conjectured in \cite{Sc83} that
any number smaller than $2$ would do.  This conjecture
was recently disproved by N.~Moshchevitin who showed
in \cite{Mo} by a nice geometric construction that the
statement of Theorem A becomes false if $\gamma$ is
replaced by any number larger than $1.947$, or more
precisely by any number larger than the largest real root
of $t^4-2t^2-4t+1$.  In this paper, we show that, in fact,
the exponent $\gamma$ in Theorem A is already optimal.

\section{Main result and notation}
 \label{main}

For any $\ux,\uy\in\bR^3$, we denote by $\ux\cdot\uy\in\bR$
their scalar product and by $\ux\wedge\uy\in\bR^3$ their
cross-product.  We also denote by $\|\ux\|=(\ux\cdot\ux)^{1/2}$
the Euclidean norm of $\ux$.  When $\ux,\uy$ are non-zero,
we further define their \emph{projective distance} by
\[
 \dist(\ux,\uy)
  = \frac{\|\ux\wedge\uy\|}{\|\ux\|\,\|\uy\|}
  = \sin(\angle(\ux,\uy))
\]
where $\angle(\ux,\uy)$ stands for the acute angle between
the lines $\bR\ux$ and $\bR\uy$ spanned by $\ux$ and $\uy$.
This number depends only on the classes of $\ux$ and $\uy$
in $\bP^2(\bR)$ and it can be shown that the induced function
$\dist\colon\bP^2(\bR)\times\bP^2(\bR)\to [0,1]$ is a true
distance function on $\bP^2(\bR)$, namely that
\[
 \dist(\ux,\uz)\le \dist(\ux,\uy)+\dist(\uy,\uz)
\]
for any $\ux,\uy,\uz\in\bR^3\setminus\{0\}$.  For a point
$\ux\in\bR^3\setminus\{0\}$ and a set $S\subset \bR^3
\setminus\{0\}$, we define
\[
 \dist(\ux,S)=\inf_{\uy\in S}\dist(\ux,\uy).
\]
In \cite{Sc76},
Schmidt proves Theorem A in a slightly more general form,
essentially equivalent to the following.

\begin{theoremB}[Schmidt]
Let $\uu,\ud\in\bR^3$ with $\uu\cdot\ud=0$.  Suppose that
$\uu$ has $\bQ$-linearly independent coordinates and that
$\ud\neq 0$.  Then, for any $\delta>0$ and any $\epsilon>0$,
there exists a non-zero point $\ux\in\bZ^3$ such that
\[
 \dist(\ux,\ud)\le \delta
 \et
 |\ux\cdot\uu| \le \epsilon \|\ux\|^{-\gamma}.
\]
\end{theoremB}

To recover Theorem A from Theorem B, it suffices to apply
the latter to the points $\uu=(1,\alpha,\beta)$ and
$\ud=(-\alpha-\beta,1,1)$ with $\delta$ fixed but small enough
so that any non-zero point $\ux\in\bZ^3$ in the cone defined by
$\dist(\ux,\ud)\le \delta$ has its last two coordinates of the
same sign.  Then, for $\ux$ as in Theorem B, the point $\pm\ux$
provides a solution of the system in Theorem A with $X=\|\ux\|$.
Moreover, by letting $\epsilon$ go to $0$, we can make $X$
arbitrarily large.

Our main result is the following.

\begin{theoremC}
Let $\psi\colon(1,\infty)\to(0,\infty)$ be an unbounded strictly
increasing function, let $\ux_0$ be a non-zero point of $\bZ^3$
and let $\delta>0$.  Then, there exist linearly
independent unit vectors $\uu$, $\uv$, $\uw$ in $\bR^3$ and
positive constants $C$, $C'$ with the following properties.
\begin{itemize}
\item[(i)] The coordinates of $\uu$ are linearly independent
     over $\bQ$.
\item[(ii)] We have $\uu\cdot\uv=\uu\cdot\uw=0$, $\dist(\ux_0,\uv)\le \delta$,
    $\dist(\ux_0,\uw)\le \delta$.
\item[(iii)] For each $X\ge 1$, there exists a non-zero point
    $\ux\in\bZ^3$ such that
    \[
    \|\ux\|\le X,\quad
    |\ux\cdot\uu|\le \frac{C}{X^{\gamma+1}},\quad
    \min\{|\ux\cdot\tuv|, |\ux\cdot\tuw|\} \le \frac{C}{X^{\gamma+1}},
    \]
    where $\tuv=\uu\wedge\uv$ and $\tuw=\uu\wedge\uw$.
\item[(iv)] For each $\ux\in\bZ^3$ with $\|\ux\|\ge C'$, we have
    $\disp
    |\ux\cdot\uu|
    \ge \frac{\dist(\ux\,,\{\uv,\uw\})}{\psi(\|\ux\|)\|\ux\|^{\gamma}}.
    $
\end{itemize}
\end{theoremC}

The following corollary shows that the conclusion of
Theorem B is best possible.

\begin{maincorollary}
Let $\psi\colon(1,\infty)\to(0,\infty)$ be an unbounded strictly
increasing function, let $\ud\in\bR^3\setminus\{0\}$ and let
$\delta>0$.  Then, there exists a unit vector $\uu\in\bR^3$
with $\bQ$-linearly independent coordinates such that
    \[
    |\ux\cdot\uu|
    \ge \frac{1}{\psi(\|\ux\|)\|\ux\|^{\gamma}}.
    \]
for any point $\ux\in\bZ^3$ of sufficiently large
norm with $\dist(\ux,\ud)>\delta$.
\end{maincorollary}

To derive the corollary from Theorem C, we simply choose a
non-zero point $\ux_0\in\bZ^3$ such that $\dist(\ux_0,\ud)\le \delta/3$
and we apply the theorem with $\delta$ replaced by $\delta/3$
and $\psi$ replaced by $(\delta/3)\psi$.  Then, with respect
to the points $\uu$, $\uv$ and $\uw$ provided by the theorem, any
non-zero point $\ux$ of $\bZ^3$ with $\dist(\ux,\ud)\ge \delta$ has
$\dist(\ux,\,\{\uv,\uw\}) \ge \delta/3$ and so, if $\|\ux\|$ is
sufficiently large, we obtain
\[
 |\ux\cdot\uu|
    \ge \frac{(\delta/3)}{(\delta/3)\psi(\|\ux\|)\|\ux\|^{\gamma}}
     = \frac{1}{\psi(\|\ux\|)\|\ux\|^{\gamma}},
\]
as requested.

Schmidt's proof of Theorem B is short and clever.  It is useful to
go back to his argument to better understand the construction that
leads to Theorem C.  Here we simply give a short account of its
first steps.  Schmidt proceeds by contradiction, assuming that,
for some positive $\delta$ and $\epsilon$, any non-zero point
$\ux\in\bZ^3$ satisfying
\begin{equation}
\label{cond_Schmidt}
 |\ux\cdot\uu|\le \epsilon \|\ux\|^{-\gamma}
\end{equation}
has $\dist(\ux,\ud)>\delta$.  Then he constructs a sequence of
rectangular parallelepipeds centered at the origin.  Each of them
has faces perpendicular to $\uu$, $\ud$ and $\uu\wedge \ud$, and
is particularly elongated in the direction of the vector $\ud$.
It has volume $8$ in order to ensure, by Minkowski's first convex
body theorem, that it contains a non-zero integer point.  Finally,
its dimensions are chosen so that all of its points $\ux$ satisfy
the condition \eqref{cond_Schmidt}.  Then the non-zero integer
points that it contains are all located in the narrow portion
of the parallelepiped defined by $\dist(\ux,\ud)> \delta$.
From this, Schmidt deduces that, for each $X\ge 1$, there exists
a non-zero point $\ux\in\bZ^3$ with $\|\ux\|\le X$ and
$|\ux\cdot\uu|\le C X^{-\gamma-1}$, where $C>0$ is
independent of $X$. This, in turn, implies the existence of
a sequence of points $(\ux_i)_{i\ge 1}$ in $\bZ^3$ with increasing
norms $X_i:=\|\ux_i\|$, such that $|\ux_i\cdot\uu|\le
CX_{i+1}^{-\gamma-1}$ for each $i\ge 1$.  This second step is
analogous to the construction of the so-called \emph{minimal points}
in \cite{DS}.  The rest of the proof uses geometry of numbers to
derive a contradiction out of these data.

To prove Theorem C, we construct a sequence of points
$(\ux_i)_{i\ge 1}$ in $\bZ^3$ which satisfies the above property
for a suitable unit vector $\uu\in\bR^3$.  Moreover, for odd
indices $i$, the class of $\ux_i$ in $\bP^2(\bR)$ converges
to that of a unit vector $\uv$ while, for even $i$, it
converges to a different class belonging to a unit vector $\uw$.
In particular, the angle between $\ux_{i-1}$ and $\ux_i$
remains bounded away from $0$.  The construction of Moshchevitin
in \cite{Mo} also shares this property, and a similar
behavior shows up in the study of extremal numbers (see
\cite[\S 4]{R11}).  One difficulty is to make $X_{i+1}$
arbitrarily large compared to $X_i$.  We will see in the next
section how it can be resolved.

Before going into this, we mention that the introduction of the
vectors $\tuv$ and $\tuw$ in part (iii) of Theorem C is simply
meant to make this statement more symmetric.  The close relation
between $\min\{|\ux\cdot\tuv|, |\ux\cdot\tuw|\}$ and
$\dist(\ux\,,\{\uv,\uw\})$ is clarified by the following estimates.

\begin{lemma}
\label{lemma:vperp}
Let $\uu,\uv,\uw$ be unit vectors in $\bR^3$ with $\uu\cdot\uv
=\uu\cdot\uw=0$.  Put $\tuv=\uu\wedge\uv$ and $\tuw=\uu\wedge\uw$.
Then for any non-zero $\ux\in\bR^3$, we have
\[
 \min\{|\ux\cdot\tuv|, |\ux\cdot\tuw|\}
 \le \|\ux\|\,\dist(\ux\,,\{\uv,\uw\})
 \le |\ux\cdot\uu| + \min\{|\ux\cdot\tuv|, |\ux\cdot\tuw|\}
\]
\end{lemma}

\begin{proof}  We find that
\[
 |\ux\cdot\tuv|^2
  = |\det(\ux,\uu,\uv)|^2
  = |(\ux\wedge\uv)\cdot\uu|^2
  = \|\ux\wedge\uv\|^2 - \|(\ux\wedge\uv)\wedge\uu\|^2
  = \|\ux\wedge\uv\|^2 - |\ux\cdot\uu|^2
\]
since $(\ux\wedge\uv)\wedge\uu = (\ux\cdot\uu)\uv-(\uv\cdot\uu)\ux
= (\ux\cdot\uu)\uv$.  As $\|\ux\wedge\uv\| = \|\ux\|\,\dist(\ux,\uv)$,
this gives
\[
 |\ux\cdot\tuv|
 \le \|\ux\|\,\dist(\ux,\uv)
 \le |\ux\cdot\uu| + |\ux\cdot\tuv|.
\]
By symmetry, the same inequality holds with $\uv$ replaced
by $\uw$ and $\tuv$ replaced by $\tuw$.  The conclusion follows.
\end{proof}

\section{The recursive step}
\label{sec:ind-step}

Our construction is based on the choice of a badly approximable
number, by which we mean a real number $\alpha$ which, for an
appropriate constant $C_1>1$, satisfies
\begin{equation}
 \label{qalpha}
  |q\alpha-p|\ge \frac{1}{C_1|q|}
\end{equation}
for each $p,q\in\bZ$ with $q\neq 0$.  We choose such a
number $\alpha$ in the interval $(0,1/2)$.  Then, the
theory of continued fractions provides sequences of integers
$(p_n)_{n\ge 1}$, $(q_n)_{n\ge 1}$ such that, for each
$n\ge 1$, we have
\begin{align}
\label{cf1} &0\le p_n\le q_n, \quad p_1=0, \quad q_1=1,\\
\label{cf2} &q_np_{n+1}-p_nq_{n+1}=(-1)^{n+1},\\
\label{cf3} &|q_n\alpha-p_n|<q_{n+1}^{-1},\\
\label{cf4} &q_n<q_{n+1}\le C_1q_n
\end{align}
(see \cite[Chapter I]{Sc80}).  The inequality $q_{n+1}\le
C_1 q_n$ follows by applying the hypothesis \eqref{qalpha}
to the left hand side of \eqref{cf3}.  For our purpose, we will
simply need \eqref{cf1}, \eqref{cf2}, \eqref{cf4} and the
following additional consequence of \eqref{qalpha} and \eqref{cf3}.

\begin{lemma}
\label{lemma:qpn-pqn}
We have $\disp |qp_n-pq_n| \ge \frac{q_n}{2C_1|q|}$
for any $p,q\in\bZ$ with $\disp 1\le |q| < q_n$.
\end{lemma}

\begin{proof} If $|q| \le q_n/\sqrt{2C_1}$, we find that
\[
 |qp_n-pq_n|
  \ge q_n|q\alpha-p|-|q||q_n\alpha-p_n|
  \ge \frac{q_n}{C_1|q|}-\frac{|q|}{q_{n+1}}
  \ge \frac{q_n}{2C_1|q|}\,.
\]
Otherwise, we have $|qp_n-pq_n|\ge 1 \ge q_n/(2C_1|q|)$
because $qp_n-pq_n$ is a non-zero integer.  This last
assertion follows from the hypothesis $1\le |q|<q_n$ together
with the fact that $p_n$ and $q_n$ are relatively prime
because of \eqref{cf2}.
\end{proof}

We say that a point $\ux$ of $\bZ^3$ is \emph{primitive}
if it is non-zero and has relatively prime coordinates.
We say that a pair $(\ux,\uy)$ of points of $\bZ^3$ is
\emph{primitive} if it satisfies the following equivalent
conditions:
\begin{itemize}
 \item[(i)] $\ux\wedge\uy$ is a primitive point of $\bZ^3$,
 \item[(ii)] $(\ux,\uy)$ is a basis of
    $\bZ^3 \cap \langle\ux,\uy\rangle_\bR$,
 \item[(iii)] there exists $\uz\in\bZ^3$ such that
     $(\ux,\uy,\uz)$ is a basis of $\bZ^3$.
\end{itemize}
With this notation at hand, the next lemma provides the
recursive step in our construction.

\begin{lemma}
 \label{lemma1}
Let $(\ux^*,\ux)$ be a primitive pair of points of $\bZ^3$,
and let $Y,X'\in\bR$ and $n\in\bZ$ with $n\ge 2$ be such that
\begin{equation}
 \label{lemma1:eq1}
 2(\|\ux^*\|+\|\ux\|) \le Y\le X'
 \et
 q_{n-1} \le \frac{2X'}{Y} < q_n\,.
\end{equation}
Then there exist $\uy,\ux'\in\bZ^3$ with the following properties:
\begin{itemize}
 \item[1)] $(\ux^*,\ux,\uy)$ is a basis of $\bZ^3$ and $(\ux,\ux')$ is a
 primitive pair of $\bZ^3$ with
 \[
  \det(\ux^*,\ux,\uy)=1,\quad
  \det(\ux^*,\ux,\ux')=q_n,\quad
  \det(\uy,\ux,\ux')=-p_n\,;
 \]
 \item[2)] $Y\le \|\uy\| \le 2Y$ and $X'\le \|\ux'\| \le 5C_1X'$\,;
 \\
 \item[3)] $\disp \dist(\ux^*,\ux')
      \le \frac{\|\ux\|}{2X'} + \frac{2C_1}{Y\|\ux^*\|\, \|\ux\| \dist(\ux^*,\ux)}$\,;
 \\
 \item[4)] the unit vector $\uu$ perpendicular to $\langle \ux^*, \ux \rangle_\bR$
   and the unit vector $\uu'$  perpendicular to $\langle \ux, \ux' \rangle_\bR$
   satisfy
 \[
   \dist(\uu,\uu')
   \le \frac{2C_1}{Y\|\ux^*\|\,\|\ux\|\dist(\ux^*,\ux)\dist(\ux,\ux')}\,.
 \]
\end{itemize}
\end{lemma}

Note that, in 4), the number $\dist(\uu,\uu')$ is independent of the choice
of $\uu$ and $\uu'$.

\begin{proof}
Put $H=\|\ux^*\wedge\ux\|$ and $\uu=H^{-1}\ux^*\wedge\ux$, so that $\uu$ is a unit
vector orthogonal to $\langle \ux^*, \ux \rangle_\bR$.  By hypothesis, there
exists $\uy_0\in\bZ^3$ such that $(\ux^*,\ux,\uy_0)$ is a basis of $\bZ^3$.
Upon writing
\[
 \uy_0=r\ux^*+s\ux+t\uu
\]
with $r,s,t\in\bR$, we find that
\[
 \pm 1 = \det(\ux^*,\ux,\uy_0) = t\det(\ux^*,\ux,\uu) = tH
\]
and so $t = \pm H^{-1}$.  Replacing $\uy_0$ by $-\uy_0$ if necessary,
we may assume that $t=H^{-1}$.  Then, replacing $\uy_0$ by $\uy_0+\ell\ux$
for a suitable $\ell\in\bZ$, we may further assume that $|s|\le 1/2$.  We define
\[
 \uy = \uy_0+a\ux^* = (a+r)\ux^*+s\ux+H^{-1}\uu \in\bZ^3
\]
where $a$ is the smallest integer for which
\[
 (a+r)\|\ux^*\| \ge Y + \frac{\|\ux\|}{2} + 1.
\]
Then, we have $\det(\ux^*,\ux,\uy)=1$ and
\[
 Y \le \|\uy\| \le Y+\|\ux^*\|+\|\ux\|+2 \le 2Y
\]
because $\|\uy-(a+r)\ux^*\| \le |s|\,\|\ux\|+H^{-1} \le \|\ux\|/2 + 1$.

We choose an integer $m$ with $|sq_n+m|\le 1/2$ and form the point
\[
 \ux' = q_n\uy + p_n\ux^* + m\ux \in \bZ^3.
\]
Since $\det(\ux^*,\ux,\uy)=1$, we find that
\[
 \det(\ux^*,\ux,\ux')=q_n \et \det(\uy,\ux,\ux')=-p_n\,.
\]
Since $\gcd(p_n,q_n)=1$, this implies that $(\ux,\ux')$ is a primitive
pair of points in $\bZ^3$.  We further note that
\begin{align*}
 \|\ux'-q_n\uy\|
 &\le p_n\|\ux^*\| + |m|\,\|\ux\| \\
 &\le q_n\|\ux^*\| + \frac{1}{2}(q_n+1)\|\ux\| \\
 &\le q_n(\|\ux^*\|+\|\ux\|) \\
 &\le \frac{1}{2}q_nY,
\end{align*}
thus $(1/2)q_nY \le \|\ux'\|\le (5/2)q_nY$.  Since
\[
 \frac{2X'}{Y} \le q_n\le C_1q_{n-1} \le \frac{2C_1X'}{Y},
\]
we conclude that $X'\le \|\ux'\| \le 5C_1X'$.  This proves 1) and 2).

To prove 3), we observe that
\begin{align*}
 \|\ux^*\wedge\ux'\|
 &= \|\ux^*\wedge(q_n\uy+m\ux)\| \\
 &= \|\ux^*\wedge((sq_n+m)\ux+q_nH^{-1}\uu)\| \\
 &\le |sq_n+m|\,\|\ux^*\wedge\ux\|+q_nH^{-1}\|\ux^*\wedge\uu\| \\
 &\le \frac{1}{2}\|\ux^*\|\,\|\ux\| + \frac{2C_1X'}{HY}\|\ux^*\|,
\end{align*}
and so
\[
 \dist(\ux^*,\ux')
 =\frac{\|\ux^*\wedge\ux'\|}{\|\ux^*\|\,\|\ux'\|}
 \le \frac{\|\ux^*\wedge\ux'\|}{\|\ux^*\|\,X'}
 \le \frac{\|\ux\|}{2X'} + \frac{2C_1}{HY}\,.
\]
Then, 3) follows because $H=\|\ux^*\|\,\|\ux\|\dist(\ux^*,\ux)$.

Put $H'=\|\ux\wedge\ux'\|$ and $\uu'=(H')^{-1}\ux\wedge\ux'$,
so that $\uu'$ is a unit vector orthogonal to
$\langle \ux, \ux' \rangle_\bR$.  We have
\[
 (\ux^*\wedge\ux)\wedge(\ux\wedge\ux')
 = \det(\ux^*,\ux,\ux') \ux = q_n\ux
\]
thus
\[
 \dist(\uu,\uu')
 =\frac{q_n\|\ux\|}{HH'}
 \le \frac{2C_1X'\|\ux\|}{YHH'}
 \le \frac{2C_1\|\ux\|\,\|\ux'\|}{YHH'}\,.
\]
This is equivalent to the estimate of 4) by definition of
$\dist(\ux^*,\ux)$ and $\dist(\ux,\ux')$.
\end{proof}

The picture below illustrates the construction of the lemma.
It shows the projection of $\bZ^3$ on the plane perpendicular
to $\ux$, with the vector $\uu$ on the horizontal line
passing through the origin.  Each dot is thus the projection
of a translate of $\bZ\ux$, and the vertical line passing
through the origin represents the projection of the subspace
$\pi = \langle \ux^*, \ux \rangle_\bR$.  The vertical line
to its left is the projection of a closest plane $\pi_1$
containing an integral point, and its distance to $\pi$ is
$H^{-1}$.

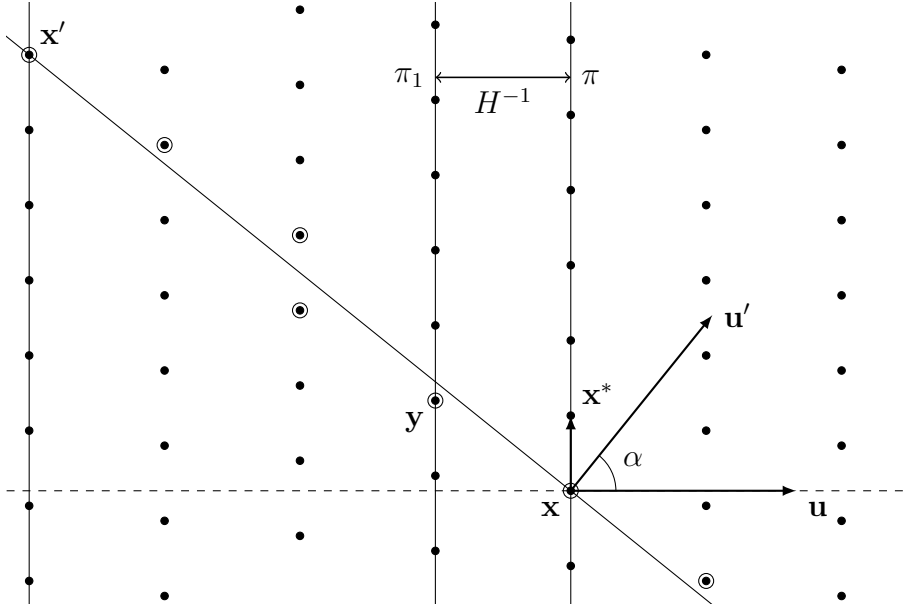
\begin{figure}[ht]
  \centering
  \begin{tikzpicture}
    \clip (-7.5,-1.5) rectangle (4.5cm,6.5cm); 
    \draw [thin, dashed] (-7.5,0) -- (5,0); 
    \foreach \x in {-7,-6,...,7}{
      \foreach \y in {-7,-6,...,7}{
        \node[draw,circle,inner sep=1pt,fill] at (1.8*\x,-0.2*\x+\y) {};
    }}
    \foreach \x in {-4,-1,0}{
      \draw[thin] (1.8*\x,-5) -- (1.8*\x,10);
    }
    \coordinate (direction) at (-1.8*4,0.2*4+5); 
    \coordinate (perp) at (0.2*4+5,1.8*4); 
    \draw [thin] ($1.1*(direction)$) -- ($-0.5*(direction)$);
    \draw node at ($(direction)$) [above right] {$\ux'$};
    \node[draw,circle,inner sep=2pt] at (-1.8*4,0.2*4+5) {};
    \node[draw,circle,inner sep=2pt] at (-1.8*3,0.2*3+4) {};
    \node[draw,circle,inner sep=2pt] at (-1.8*2,0.2*2+2) {};
    \node[draw,circle,inner sep=2pt] at (-1.8*2,0.2*2+3) {};
    \node[draw,circle,inner sep=2pt] at (-1.8*1,0.2*1+1) {};
    \node[draw,circle,inner sep=2pt] at (-1.8*0,0.2*0+0) {};
    \node[draw,circle,inner sep=2pt] at (1.8*1,-0.2*1-1) {};
    \draw [thick,-latex] (0,0)
        -- (0,1) node [above right] {$\ux^*$};
    \draw [thick,-latex] (0,0)
        -- (3,0) node [below right] {$\uu$};
    \draw node at (0,0) [below left] {$\ux$};
    \draw node at (-1.8,1.2) [below left] {$\uy$};
    \draw [thick,-latex] (0,0)
        -- ($3/9.24*(perp)$) node [right] {$\uu'$};
    \draw (0.6,0) arc (0:51:0.6);
    \draw node at (0.55,0.65) [below right] {$\alpha$}; 
    \draw[<->,semithick] (-1.8,5.5)--(0,5.5);
    \draw node at (-0.9,5.5) [below] {$H^{-1}$};
    \draw node at (0,5.5) [right] {$\pi$};
    \draw node at (-1.8,5.5) [left] {$\pi_1$};
  \end{tikzpicture}
  \label{figure}
\end{figure}

The vector $\uu'$ is obtained by rotating $\uu$ by a small
angle $\alpha$ about the line $\bR\ux$.  Then, in each plane
parallel to $\pi$, the integer points $\uz$ for which
$|\uz\cdot\uu'|$ is minimal form at most two translates
of $\bZ\ux$.  They are shown on the picture as circled dots.
The exact angle $\alpha$ is obtained through a process which
is similar to that of fine tuning the focus of a microscope.
A first coarse adjustment is to choose $\alpha$ so that,
on $\pi_1$, the minimal value for $|\uz\cdot\uu'|$ is
obtained at $\uz=\uy$ or at $\uz=\uy+\ux^*$ where $\uy$
is an integer point of $\pi_1$ of norm about $Y$ which is
essentially closest to the line $\bR\ux^*$.  The finer
adjustment consists in choosing $\alpha$ so that the plane
perpendicular to $\uu'$ contains a non-zero integer point
$\ux'$ of arbitrarily large norm (about $X'$) also pointing
in a direction close to that of $\ux^*$.  However, the
most important feature of this correction, which is fundamental
for the proof of Lemma \ref{lemma3} below, is that the minimal
values for $|\uz\cdot\uu'|$ are essentially equidistributed
as we move along the relevant planes parallel to $\pi$,
from the one containing $\ux'$ to the one containing $-\ux'$.

\section{The main construction}

We now apply Lemma \ref{lemma1} to produce sequences of
points of the sort that we need for the proof of Theorem C.

\begin{lemma}
\label{lemma2}
Let $(\ux_0,\ux_1)$ be a primitive pair of points in $\bZ^3$, and
let $(X_i)_{i\ge 0}$ be a  sequence of positive real numbers
satisfying $X_0=\|\ux_0\|$, $X_1=\|\ux_1\|$ and
\begin{equation}
 \label{lemma2:eq1}
 12C_1X_i \le X_i^\gamma \le X_{i+1},
 \quad
 2X_{i+1}^2 \le X_iX_{i+2}
 \quad\text{for each $i\ge 1$.}
\end{equation}
Suppose further that $5X_0\le X_1$ and that
\begin{equation}
 \label{lemma2:eq2}
 \frac{5C_1X_1}{X_2} + \frac{4C_1}{\delta_0X_0X_1^{\gamma+1}}
 \le \delta_0
 \quad\text{where}
 \quad
 \delta_0=\frac{1}{2}\dist(\ux_0,\ux_1).
\end{equation}
Then there exist linearly independent unit vectors $\uu$, $\uv$, $\uw$
in $\bR^3$ and sequences $(\ux_i)_{i\ge 2}$,  $(\uy_i)_{i\ge 1}$ in
$\bZ^3$ which, for each $i\ge 1$, satisfy the following properties:
\begin{itemize}
 \item[1)] letting $n\ge 2$ denote the integer for which $q_{n-1}\le
 2X_{i+1}/X_i^{\gamma} < q_n$, we have
 \[
  \det(\ux_{i-1},\ux_i,\uy_i)=1, \quad
  \det(\ux_{i-1},\ux_i,\ux_{i+1})=q_n, \quad
  \det(\uy_i,\ux_i,\ux_{i+1})=-p_n\,;
 \]
 \item[2)] $X_i^\gamma\le \|\uy_i\| \le 2X_i^\gamma$ and
 $X_{i+1}\le \|\ux_{i+1}\| \le 5C_1X_{i+1}$\,;
 \smallskip
 \item[3)] $\dist(\ux_{i-1},\ux_i)\ge\delta_0$ and
 \[
  \delta_0
  \ge \frac{5C_1X_i}{X_{i+1}}
         + \frac{4C_1}{\delta_0 X_{i-1} X_i^{\gamma+1}}
  \ge \begin{cases}
       \dist(\ux_{i-1},\uv) &\text{if $i$ is even,}\\[2pt]
       \dist(\ux_{i-1},\uw) &\text{if $i$ is odd;}
      \end{cases}
 \]
 \item[4)] the unit vector $\uu_i$ perpendicular to
   $\langle \ux_{i-1}, \ux_i \rangle_\bR$ satisfies
 \[
  \dist(\uu_i,\uu)
      \le \frac{4C_1}{\delta_0^2 X_{i-1} X_i^{\gamma+1}}\,.
 \]
\end{itemize}
Moreover, we have $\uu\cdot\uv=\uu\cdot\uw=0$, $\dist(\ux_0,\uv)
\le 3\delta_0$ and $\dist(\ux_0,\uw)\le \delta_0$.
\end{lemma}

\begin{proof}
Starting from the primitive pair $(\ux_0,\ux_1)$, Lemma \ref{lemma1} allows
us to construct recursively pairs of vectors $(\uy_1,\ux_2), (\uy_2,\ux_3), \dots$
which for each $i\ge 1$ fulfill the conditions 1) and 2) of Lemma \ref{lemma2}
as well as
\begin{itemize}
 \item[3')] $\disp \dist(\ux_{i-1},\ux_{i+1})
   \le \frac{5C_1X_i}{2X_{i+1}} + \frac{2C_1}{X_{i-1}X_i^{\gamma+1}\dist(\ux_{i-1},\ux_i)}$,
 \smallskip
 \item[4')] $\disp \dist(\uu_i,\uu_{i+1})
   \le \frac{2C_1}{X_{i-1}X_i^{\gamma+1}\dist(\ux_{i-1},\ux_i)\dist(\ux_i,\ux_{i+1})}$.
\end{itemize}
To construct $(\uy_i,\ux_{i+1})$ from the preceeding pairs,
we simply apply Lemma \ref{lemma1} with $\ux^*=\ux_{i-1}$, $\ux=\ux_i$,
$Y=X_i^\gamma$ and $X'=X_{i+1}$.  Then we define $\uy_i=\uy$ and $\ux_{i+1}=\ux'$.
The lemma applies because, at each step the pair $(\ux_{i-1},\ux_i)$ is
primitive and
\[
 2(\|\ux_{i-1}\|+\|\ux_i\|)
   \le 10C_1(X_{i-1}+X_i)
   \le 12C_1X_i
   \le X_i^\gamma
   \le X_{i+1}\,.
\]

Define
\[
 \delta_i := \frac{5C_1X_i}{2X_{i+1}} + \frac{2C_1}{\delta_0 X_{i-1}X_i^{\gamma+1}}
 \quad
 \text{for each $i\ge 1$.}
\]
Then the hypotheses \eqref{lemma2:eq1} and \eqref{lemma2:eq2} imply that
$\delta_i\le \delta_{i-1}/2$ for each $i\ge 1$.  We claim that
\begin{equation}
 \label{lemma2:eq3}
 \dist(\ux_{i-1},\ux_i) \ge \delta_0+\delta_{i-1}
 \et
 \dist(\ux_{i-1},\ux_{i+1}) \le \delta_i
 \quad
 \text{for each $i\ge 1$.}
\end{equation}
For $i=1$, this is clear because $\dist(\ux_0,\ux_1)=2\delta_0$ and so
3') yields $\dist(\ux_0,\ux_2)\le \delta_1$.  Moreover, if \eqref{lemma2:eq3}
holds for some $i\ge 1$, then
\[
 \dist(\ux_i,\ux_{i+1})
 \ge \dist(\ux_i,\ux_{i-1})-\dist(\ux_{i-1},\ux_{i+1})
 \ge \delta_0+\delta_{i-1}-\delta_i
 \ge \delta_0+\delta_i
\]
and so 3') with $i$ replaced by $i+1$ yields $\dist(\ux_i,\ux_{i+2}) \le \delta_{i+1}$.

In view of the second part of \eqref{lemma2:eq3}, we conclude that, for even $i\ge 2$,
the class of $\ux_{i-1}$ in $\bP^2(\bR)$ converges to the class of a unit vector $\uv$
with
\[
 \dist(\ux_{i-1},\uv)
 \le \sum_{j=0}^\infty \dist(\ux_{i+2j-1},\ux_{i+2j+1})
 \le \sum_{j=0}^\infty \delta_{i+2j}
 \le 2\delta_i.
\]
Similarly, for odd $i\ge 1$, the class of $\ux_{i-1}$ converges to the
class of a unit vector $\uw$ with $\dist(\ux_{i-1},\uw) \le 2\delta_i$.
This proves 3) and implies in particular that $\dist(\ux_0,\uw)\le
\delta_0$ and that
\[
 \dist(\ux_0,\uv)\le \dist(\ux_0,\ux_1)+\dist(\ux_1,\uv) \le 3\delta_0.
\]

The fact that $\dist(\ux_{i-1},\ux_i) \ge \delta_0$ for each $i\ge 1$
combined with 4') yields
\[
 \dist(\uu_i,\uu_{i+1})
 \le \frac{2C_1}{\delta_0^2X_{i-1}X_i^{\gamma+1}}
 \quad
 \text{for each $i\ge 1$.}
\]
From this we conclude that the class of $\uu_i$ in $\bP^2(\bR)$
converges to that of a unit vector $\uu\in\bR^3$ with
\[
 \dist(\uu_i,\uu) \le \frac{4C_1}{\delta_0^2X_{i-1}X_i^{\gamma+1}}
 \quad
 \text{for each $i\ge 1$.}
\]
This proves 4).  Moreover, the relations $\uu_i\cdot\ux_i=0$ ($i\ge 1$) imply
by continuity that $\uu\cdot\uv=\uu\cdot\uw=0$.
\end{proof}

\begin{lemma}
\label{lemma3}
In the context of Lemma \ref{lemma2}, suppose furthermore that
\begin{equation}
 \label{lemma3:eq0}
  X_{i+1} \ge X_{i-1}X_i^{\gamma+2}
  \quad
  \text{for each $i\ge 1$.}
\end{equation}
If the product $\delta_0^2 X_1$ is larger than a suitable function
of $C_1$, then, for each index $i\ge 1$ and each
$\ux\in\bZ^3$ with
\begin{equation}
 \label{lemma3:eq1}
 \frac{X_i}{X_1} \le \|\ux\| < \frac{X_{i+1}}{X_1},
\end{equation}
we have
\[
 |\ux\cdot\uu|
 \ge \frac{1}{X_1^3 X_{i-1}\|\ux\|^{\gamma}}
     \begin{cases}
        1 &\text{if \ $\ux\notin \langle \ux_{i-1},\ux_i\rangle_\bZ$,}\\[5pt]
        \dist(\ux,\{\uv,\uw\}) &\text{if \ $\ux\in \langle\ux_{i-1},\ux_i\rangle_\bZ$.}
     \end{cases}
\]
\end{lemma}

In the proof below, we use repeatedly the estimates 1), \ldots, 4)
of lemma \ref{lemma2}.  To alleviate the exposition, we simply
refer to them as 1), \ldots, 4) respectively.  We also put a
star on the right of an inequality sign to mean that this
inequality holds when $\delta_0^2 X_1$ is sufficiently large,
with a lower bound depending only on $C_1$.

\begin{proof}
For any integer $j\ge 1$ and any $\ux\in\bZ^3$, we have
\[
 |\ux\cdot\uu_j|
   = \frac{|\det(\ux,\ux_{j-1},\ux_j)|}{\|\ux_{j-1}\wedge\ux_j\|}
   \ge \frac{|\det(\ux,\ux_{j-1},\ux_j)|}{\|\ux_{j-1}\|\,\|\ux_j\|}
   \ge \frac{|\det(\ux,\ux_{j-1},\ux_j)|}{(5C_1)^2 X_{j-1} X_j}
\]
by 2) and the fact that $\|\ux_i\|=X_i$ when $i\le 1$. Since
$\uu_j$ and $\uu$ are unit vectors, we also have
\[
 |\ux\cdot(\uu_j-\uu)|
   \le \|\ux\|\,\|\uu_j-\uu\|
   \le 2\|\ux\|\dist(\uu_j,\uu)
   \le \frac{8C_1\|\ux\|}{\delta_0^2 X_{j-1} X_j^{\gamma+1}}
\]
by 4). Combining the last two estimates, we get
\begin{equation}
\label{x.u}
 |\ux\cdot\uu|
   \ge |\ux\cdot\uu_j| - |\ux\cdot(\uu_j-\uu)| \\
   \ge \frac{|\det(\ux,\ux_{j-1},\ux_j)|}{(5C_1)^2 X_{j-1} X_j}
       - \frac{8C_1\|\ux\|}{\delta_0^2 X_{j-1} X_j^{\gamma+1}}.
\end{equation}

Now, assume that $\ux$ satisfies the condition \eqref{lemma3:eq1}
for some $i\ge 1$. Since $(\ux_{i-1},\ux_i,\uy_i)$ is a basis
of $\bZ^3$, we may write
\begin{equation}
\label{x}
 \ux = q\uy_i + p\ux_{i-1} + r\ux_i
\end{equation}
with $p,q,r \in \bZ$.  By 1), this yields
\begin{equation}
\label{detxxi-1xi}
 \det(\ux,\ux_{i-1},\ux_i) = q\det(\ux_{i-1},\ux_i,\uy_i) = q\,,
\end{equation}
as well as
\begin{equation}
\label{detxxixi+1}
 \det(\ux,\ux_i,\ux_{i+1})
   = q\det(\uy_i,\ux_i,\ux_{i+1}) + p\det(\ux_{i-1},\ux_i,\ux_{i+1})
   = -(qp_n-pq_n),
\end{equation}
where $n\ge 2$ is the integer for which
\begin{equation}
\label{defn}
 q_{n-1} \le \frac{2X_{i+1}}{X_i^\gamma} < q_n.
\end{equation}
In the computations below, we use either \eqref{detxxi-1xi}
combined with \eqref{x.u} for $j=i$, or \eqref{detxxixi+1}
combined with \eqref{x.u} for $j=i+1$.  We distinguish four cases.

\textbf{Case 1.}
Suppose first that $\|q\uy_i\| \ge C_2\|\ux\|$ where
$C_2 = (8C_1)^3/\delta_0^2$.  Then, using the upper bound
for $\|\uy_i\|$ provided by 2), we find that
\[
 |q| = \frac{\|q\uy_i\|}{\|\uy_i\|}
  \ge \frac{C_2\|\ux\|}{2 X_i^{\gamma}}.
\]
Applying \eqref{x.u} with $j=i$, the above inequality
combined with \eqref{detxxi-1xi} yields
\[
 |\ux\cdot\uu|
   \ge \frac{C_2\|\ux\|}{2(5C_1)^2 X_{i-1} X_i^{\gamma+1}}
       - \frac{8C_1\|\ux\|}{\delta_0^2 X_{i-1} X_i^{\gamma+1}}
   \ge \frac{2C_1\|\ux\|}{\delta_0^2 X_{i-1} X_i^{\gamma+1}}.
\]
Using the hypothesis $X_i \le X_1\|\ux\|$ to eliminate $X_i$
from the last estimate, we conclude that
\[
 |\ux\cdot\uu|
   \ge \frac{2C_1}{\delta_0^2 X_1^{\gamma+1} X_{i-1} \|\ux\|^{\gamma}}
   \ge \frac{1}{X_1^3 X_{i-1} \|\ux\|^{\gamma}}.
\]

\textbf{Case 2.}
Suppose now that $\disp 0 <\|q\uy_i\| < C_2\|\ux\|$ where
$C_2$ is as in Case 1.  For the integer $n$ defined by \eqref{defn}
we find, using 2), that
\[
 q_n\|\uy_i\| \ge q_n X_i^\gamma \ge 2X_{i+1},
\]
and thus
\begin{equation}
\label{Case2:eq1}
 \frac{|q|}{q_n}
   = \frac{\|q\uy_i\|}{q_n\|\uy_i\|}
   < \frac{C_2\|\ux\|}{2X_{i+1}}.
\end{equation}
Using $\|\ux\| < X_1^{-1}X_{i+1}$, this gives
\[
 |q| < \frac{C_2}{2X_1}q_n \lee q_n.
\]
Then, by Lemma \ref{lemma:qpn-pqn}, we conclude that
\[
 |qp_n-pq_n|
   \ge \frac{q_n}{2C_1|q|}
   \ge \frac{X_{i+1}}{C_1C_2\|\ux\|}\,,
\]
where the second estimate comes from \eqref{Case2:eq1}.
Applying \eqref{x.u} with $j=i+1$, the above inequality
combined with \eqref{detxxixi+1} yields
\[
 |\ux\cdot\uu|
   \ge \frac{1}{C_3 X_i \|\ux\|}
       - \frac{8C_1\|\ux\|}{\delta_0^2 X_i X_{i+1}^{\gamma+1}}
 \quad
 \text{where}
 \quad
 C_3 = 25 C_1^3 C_2.
\]
Using $X_{i+1} > X_1 \|\ux\|$, this in turn yields
\[
 |\ux\cdot\uu|
   \ge \frac{1}{C_3 X_i \|\ux\|}
       - \frac{8C_1}{\delta_0^2 X_1^{\gamma+1} X_i \|\ux\|^\gamma}\\
   \gee \frac{1}{2C_3 X_i \|\ux\|}.
\]
Finally, since $q\neq 0$, we have $C_2\|\ux\| > \|q\uy_i\|
\ge \|\uy_i\| \ge X_i^\gamma$ by 2).  Using this to eliminate
$X_i$ from the previous inequality, we conclude that
\[
 |\ux\cdot\uu|
   \ge \frac{1}{2C_3 C_2^{1/\gamma} \|\ux\|^\gamma}
   \gee \frac{1}{X_1^3 \|\ux\|^\gamma}.
\]

The above two cases cover the situation where $\ux\notin
\langle \ux_{i-1},\ux_i\rangle_\bZ$. The next two cases
complete the proof of the lemma when $\ux\in
\langle \ux_{i-1},\ux_i\rangle_\bZ$.

\textbf{Case 3.}  Suppose that $q=0$ but $p\neq 0$.  Then,
the inequality \eqref{x.u} with $j=i+1$ combined with
\eqref{detxxixi+1} yields
\[
 |\ux\cdot\uu|
   \ge \frac{|p|q_n}{(5C_1)^2 X_i X_{i+1}}
      - \frac{8C_1\|\ux\|}{\delta_0^2 X_i X_{i+1}^{\gamma+1}}.
\]
Using the lower bound for $q_n$ given by \eqref{defn}
and the hypothesis that $\|\ux\|\le X_1^{-1}X_{i+1}$,
we deduce that
\begin{equation}
\label{Case3:eq1}
 |\ux\cdot\uu|
   \ge \frac{2|p|}{(5C_1)^2 X_i^{\gamma+1}}
       - \frac{8C_1}{\delta_0^2 X_1 X_i X_{i+1}^{\gamma}}
   \gee \frac{|p|}{(5C_1)^2 X_i^{\gamma+1}}.
\end{equation}
We also note that
\[
 \dist(\ux,\ux_i)
  = \frac{\|\ux\wedge\ux_i\|}{\|\ux\|\,\|\ux_i\|}
  = \frac{|p|\,\|\ux_{i-1}\wedge\ux_i\|}{\|\ux\|\,\|\ux_i\|}
  \le \frac{|p|\,\|\ux_{i-1}\|}{\|\ux\|}
  \le \frac{5C_1 |p|\,X_{i-1}}{\|\ux\|}
\]
where the last step uses 2) and the fact that
$\|\ux_{i-1}\|=X_{i-1}$ when $i\le 2$.  Now, consider
the main estimate of 3) with $i$ replaced by $i+1$.  Applying
the hypothesis \eqref{lemma3:eq0} with $i$ replaced by $i+1$ 
to eliminate $X_{i+2}$, it gives
\begin{equation}
\label{Case3:eq2}
 \dist(\ux_i,\{\uv,\uw\})
    \le \frac{5C_1X_{i+1}}{X_{i+2}}
        + \frac{4C_1}{\delta_0 X_i X_{i+1}^{\gamma+1}}
    \le \frac{9C_1}{\delta_0 X_i X_{i+1}^{\gamma+1}}\,.
\end{equation}
Since $\|\ux\|\le X_1^{-1}X_{i+1}$, the latter two estimates yield
\[
 \dist(\ux,\{\uv,\uw\})
  \le \dist(\ux,\ux_i)+\dist(\ux_i,\{\uv,\uw\})
  \lee \frac{6C_1|p|\,X_{i-1}}{\|\ux\|}.
\]
We view this as a lower bound for $|p|$.  Substituting
it into \eqref{Case3:eq1} and then using $X_i\le X_1\|\ux\|$
to eliminate $X_i$, we find
\[
 |\ux\cdot\uu|
   \ge \frac{\dist(\ux,\{\uv,\uw\})\|\ux\|}{(6C_1)^3 X_{i-1} X_i^{\gamma+1}}
   \ge \frac{\dist(\ux,\{\uv,\uw\})}{(6C_1)^3 X_1^{\gamma+1} X_{i-1} \|\ux\|^{\gamma}}
   \gee \frac{\dist(\ux,\{\uv,\uw\})}{X_1^3 X_{i-1} \|\ux\|^{\gamma}}.
\]

\textbf{Case 4.}  Finally, suppose that $p = q = 0$.
By 2) and \eqref{lemma3:eq0}, we have 
$\|\ux_{i-1}\|\lee X_1X_{i-1} \le X_1^{-1}X_{i+1}$.  Thus, 
the estimate \eqref{Case3:eq1} of Case 3 applies in
particular to the choice of $\ux=\ux_{i-1}$ (corresponding
to $p=1$ and $q=r=0$).  This gives
\[
 |\ux_{i-1}\cdot\uu|
   \ge \frac{1}{(5C_1)^2 X_i^{\gamma+1}}\,.
\]
In the present case, we have $\ux=r\ux_i$ with $r\neq 0$.  
Applying the above inequality with $i$ replaced by
$i+1$ and then using 2), we thus find that
\[
 |\ux\cdot\uu|
   = |r\ux_i\cdot\uu|
   \ge \frac{|r|}{(5C_1)^2 X_{i+1}^{\gamma+1}}
   = \frac{\|\ux\|}{(5C_1)^2 X_{i+1}^{\gamma+1}\|\ux_i\|}
   \ge \frac{\|\ux\|}{(5C_1)^3 X_i X_{i+1}^{\gamma+1}}\,.
\]
By \eqref{Case3:eq2}, we also have that
\[
 \dist(\ux,\{\uv,\uw\})
  = \dist(\ux_i,\{\uv,\uw\})
  \le \frac{9C_1}{\delta_0 X_i X_{i+1}^{\gamma+1}}\,,
\]
and so
\[
 |\ux\cdot\uu|
   \ge \frac{\delta_0\|\ux\|}{(6C_1)^4} \dist(\ux,\{\uv,\uw\})
   \gee \frac{\dist(\ux,\{\uv,\uw\})}{X_1^3 X_{i-1} \|\ux\|^{\gamma}}\,.
\qedhere
\]
\end{proof}

\section{Proof of Theorem C}
\label{sec:proof}

Without loss of generality, we may assume that $\ux_0$
is primitive.  Then, it belongs to a
basis $(\ux_0,\ux'_0,\ux''_0)$ of $\bZ^3$.  Upon
replacing $\ux'_0$ by $\ux'_0+m\ux_0$ for a sufficiently
large positive integer $m$ if necessary, we may further assume
that $\dist(\ux_0,\ux'_0)\le \delta/2$.  We define
\begin{equation}
 \label{def_delta0}
 \ux_1=n\ux'_0+\ux_0,
 \quad
 X_0=\|\ux_0\|,
 \quad
 X_1=\|\ux_1\|
 \et
 \delta_0=\frac{1}{2}\dist(\ux_0,\ux_1),
\end{equation}
where $n$ is a positive integer to be determined. Since
\[
 \lim_{n\to\infty} X_1 =\infty
 \et
 \lim_{n\to\infty} \delta_0
   = \frac{1}{2}\dist(\ux_0,\ux'_0)
   \le \frac{\delta}{4}\,,
\]
we can choose $n$ so that the following conditions hold:
\begin{itemize}
\item $X_1\ge \max\{5X_0, (12C_1)^\gamma\}$ and $\delta_0\le \delta/3$,
\item the inequality \eqref{lemma2:eq2} of Lemma \ref{lemma2} holds
  for any choice of $X_2$ with $X_2\ge X_1^\gamma$,
\item the product $\delta_0^2X_1$ is larger than the
  function of $C_1$ involved in Lemma \ref{lemma3}.
\end{itemize}
Then, for each $i\ge 1$, we choose recursively a value for
$X_{i+1}$ so that
\[
 X_{i+1}\ge X_{i-1}X_i^{\gamma+2}
 \et
 \psi\left(\frac{X_{i+1}}{X_1}\right) \ge X_1^3X_{i}.
\]
This ensures that all the hypotheses of Lemmas \ref{lemma2} and
\ref{lemma3} are satisfied.  We claim that the vectors $\uu$,
$\uv$ and $\uw$ provided by Lemma \ref{lemma2} satisfy all the
conditions of the theorem.

First of all, the last assertion of Lemma \ref{lemma2} together
with the fact that $3\delta_0\le \delta$ implies that
condition (ii) is fulfilled.

In order to verify condition (iii), we may restrict to the
values of $X$ with $X\ge 5C_1X_0$.  For such $X$, there exists
a unique index $i\ge 1$ such that
\[
 5C_1X_{i-1}\le X < 5C_1X_i.
\]
Then the point $\ux:=\ux_{i-1}$ has the required properties.
It satisfies $\|\ux_{i-1}\| \le 5C_1X_{i-1}\le X$ and
\[
 |\ux_{i-1}\cdot\uu|
  = |\ux_{i-1}\cdot(\uu-\uu_i)|
  \le \|\ux_{i-1}\|\,\|\uu-\uu_i\|
  \le 2\|\ux_{i-1}\|\dist(\uu_i,\uu).
\]
Using parts 2) and 4) of Lemma \ref{lemma2} together with
the fact that $X_i \ge X/(5C_1)$, this becomes
\[
 |\ux_{i-1}\cdot\uu|
  \le 2(5C_1X_{i-1})\frac{4C_1}{\delta_0^2 X_{i-1} X_i^{\gamma+1}}
  \le \frac{C_4}{X^{\gamma+1}}\,,
\]
where $C_4=(6C_1)^5\delta_0^{-2}$.  Combining part 3)
of Lemma \ref{lemma2} with the hypothesis that $X_{i+1}\ge
X_{i-1}X_i^{\gamma+2}$, we also find that
\[
 \dist(\ux_{i-1},\{\uv,\uw\})
   \le \frac{5C_1X_i}{X_{i+1}}
       + \frac{4C_1}{\delta_0 X_{i-1} X_i^{\gamma+1}}
   \le \frac{9C_1}{\delta_0 X_{i-1} X_i^{\gamma+1}},
\]
and thus, using Lemma \ref{lemma:vperp}, we obtain
\[
 \min\{|\ux_{i-1}\cdot\tuv|, |\ux_{i-1}\cdot\tuw|\}
   \le \|\ux_{i-1}\| \dist(\ux_{i-1},\{\uv,\uw\})
   \le \frac{45 C_1^2}{\delta_0 X_i^{\gamma+1}}
   \le \frac{C_4}{X^{\gamma+1}}.
\]

To prove condition (iv), choose an arbitrary $\ux\in\bZ^3$
with $\|\ux\| \ge X_2/X_1$. Then, there exists an index
$i\ge 2$ such that $X_1^{-1}X_i \le \|\ux\| < X_1^{-1}X_{i+1}$.
By virtue of the choice of the sequence $(X_i)_{i\ge 0}$, we have
\[
 \psi(\|\ux\|) \ge \psi\left(\frac{X_{i}}{X_1}\right) \ge X_1^3X_{i-1}
\]
and so Lemma \ref{lemma3} gives
\[
 |\ux\cdot\uu|
  \ge \frac{\dist(\ux,\,\{\uv,\uw\})}{X_1^3X_{i-1}\|\ux\|^\gamma}
  \ge \frac{\dist(\ux,\,\{\uv,\uw\})}{\psi(\|\ux\|)\|\ux\|^\gamma}.
\]
This proves  condition (iv) with $C'=X_2/X_1$.

Finally, to prove condition (i), suppose on the contrary
that there exists a non-zero point $\ux\in\bZ^3$ such that
$\ux\cdot\uu=0$.  For any sufficiently large index $i$,
there is an integer multiple of $\ux$ with norm between
$X_i/X_1$ and $X_{i+1}/X_1$.  Then Lemma \ref{lemma3}
shows that $\ux\in\langle \ux_{i-1},\ux_i\rangle_\bQ$
and so $\ux\in\langle \ux_{i-1},\ux_i\rangle_\bZ$ because
$(\ux_{i-1},\ux_i)$ is a primitive pair.  Since, by part 1)
of Lemma \ref{lemma2}, any
triple $(\ux_{i-1},\ux_i,\ux_{i+1})$ is linearly independent,
this implies that
\[
 \ux \in \langle \ux_{i-1},\ux_i\rangle_\bZ
      \cap \langle \ux_i,\ux_{i+1}\rangle_\bZ
 = \bZ\ux_i
\]
for each large enough $i$, a contradiction.


\end{document}